\documentclass[12pt]{article}
\usepackage{amssymb,amsmath,amsthm,amsfonts,amscd}
\usepackage{graphicx}
\usepackage{color}
\usepackage[all]{xy}
\usepackage{hyperref}
\hypersetup{
    colorlinks=true,
    linkcolor=blue,
    citecolor=blue,
    filecolor=blue,
    urlcolor=blue
}
\newtheorem{myth}{Theorem}[section]
\newtheorem{mylem}{Lemma}[section]
\newtheorem{mypro}{Proposition}[section]
\newtheorem{mydef}{Definition}[section]

\def\XXint#1#2#3{{\setbox0=\hbox{$#1{#2#3}{\int}$}
    \vcenter{\hbox{$#2#3$}}\kern-.5\wd0}}

\def\YYint#1#2#3{{\setbox0=\hbox{$#1{#2#3}{\int}$}
    \lower1ex\hbox{$#2#3$}\kern-.46\wd0}}
\def\YYYint#1#2#3{{\setbox0=\hbox{$#1{#2#3}{\int}$}
    \lower0.35ex\hbox{$#2#3$}\kern-.48\wd0}}

\def\ZZint#1#2#3{{\setbox0=\hbox{$#1{#2#3}{\int}$}
    \raise1.15ex\hbox{$#2#3$}\kern-.57\wd0}}
\def\ZZZint#1#2#3{{\setbox0=\hbox{$#1{#2#3}{\int}$}
    \raise0.85ex\hbox{$#2#3$}\kern-.53\wd0}}

\begin{document}
\title{Euler-Lagrangian approach to  stochastic Euler equations in Sobolev Spaces}

\author{ Juan Londo\~no and Christian Olivera}

\date{}

\maketitle

\noindent \textit{ {\bf Key words and phrases:} 
Stochastic partial differential equations, Euler equation,  Lagragian formulation, Ito-Wentzell-Kunita Formula, Sobolev Spaces .}

\vspace{0.3cm} \noindent {\bf MSC2010 subject classification:} 60H15, %SPDE
 35R60, %SPDE
 35F10, %Initial conditions for 1-order lin DE
 60H30. %applications of Stoch. An. 

%%%%%%%%%%%%%%%%%%%%%%%%%%%%%%
\begin{abstract}

The purpose of this paper is to establish the equivalence between Lagrangian and  classical formulations for the stochastic  incompressible Euler equations, the proof 
is based in  Ito-Wentzell-Kunita formula and stochastic analysis techniques.  Moreover, we prove a local
existence result  for the Lagrangian  formulation in  suitable  Sobolev Spaces. 
 \end{abstract}
%%%%%%%%%%%%%%%%%%%%%%%%%%%%%%

\maketitle

%\renewcommand{\thefootnote}{\fnsymbol{footnote}}
%
%\renewcommand{\thefootnote}{\fnsymbol{footnote}}
%
%\title{ Discretization by euler's method for regular lagrangian flow.  }
%\author{Christian Olivera $^{1,}$\footnote{ supported by FAPESP 
%		by the grants 2017/17670-0 and 2015/07278-0, by  CNPq by the grant
%		426747/2018-6.} \hskip0.2cm 
%Juan D. Londo\~no $^{2}$ \vspace*{0.1in} \\
%$^{1}$ Departamento de Matem\'atica, Universidade Estadual de Campinas,\\
%13.081-970-Campinas-SP-Brazil. \\
%colivera@ime.unicamp.br \vspace*{0.1in} \\
% $^{2}$ \quad  j209372@dac.unicamp.br \vspace*{0.1in}}
%
%
%
%
%
%
%
%\date{}
%\maketitle
%\thispagestyle{empty}
%\vspace{-10pt}
%%%%%%%%%%%%%%%%%%%%%%%%%%%%%%%%%%%%%%%%%%%%%%%%%%%%%%%%%%%%%
%\begin{abstract}
%
%	
% This paper is concerned with the numerical analysis of the explicit Euler scheme
% for  ordinary differential equations with non-Lipschitz vector fields. 
%We prove the convergence of the Euler scheme to regular  lagrangian flow (Diperna-Lions flows)
%which is the right concept of the solution in this context. Moreover, we show that rate of convergence  is $\frac{1}{2}$. 
%\end{abstract}
%
%
%
%
%{\bf MSC 2010\/}: Primary 	65L05  : Secondary 	34C99 .

 \smallskip

%{\bf Key Words and Phrases}:  Euler scheme,  regular  lagrangian flow, non-regular coefficients, numerical approximation.

\section{Introduction}

We study a  Lagragian formulation (following \cite{Constantin}, \cite{Flandoli} and \cite{Pooley} ) of the incompressible Euler equations on a domain
$\mathbb{T}^{d}$ . The  Euler equations with transport noise model the flow of an incompressible inviscid fluid and are (classically) formulated in terms of a divergence--free vector field u (i.e.  $\nabla \cdot u=0$) as follows:

\begin{equation}\label{Euler}
du_{t} + (u_{t}  \cdot \nabla u_{t}+ \nabla p_{t} ) dt + \sum_{k} \mathcal{L}_{\sigma^{k}}^{\ast} u_{t} \  \circ dW_{t}^{k}=0
\end{equation}

where $p$ is a scalar potential representing internal pressure, $\mathcal{L}_{\sigma^{j}}^{\ast}u:=\sigma^{j}\cdot\nabla u+\left(\nabla\sigma^{j}\right)^{\ast} u$
($\mathcal{L}$ is the Lie derivative ), $W_{t}^{k}$ is a Wiener process and the integration is in the Stratonovich sense. The divergence-free condition reflects the incompressibility constraint.
 Equations related to
fluid dynamics with multiplicative noise appeared in several other works, see for instance \cite{Alonso1}, \cite{Alonso2}
\cite{Brzeniak}, \cite{Falkovich},\cite{Flandoli}, \cite{Flandoli2}, \cite{Flandoli3}, \cite{Flandoli4}  and many others.

The main topic of this work, namely the Euler-Lagrangian formulation, called also Constantin-Iyer
representation after \cite{Constantin}, \cite{Constantin2},  among related works, see for instance \cite{Besse},  \cite{Fang}, \cite{Ledesma}, \cite{Olivera}, \cite{Rezakhanlou}. 
First we  show the Euler-Lagrangian formulation   is  equivalent to
the stochastic  Euler equations (\ref{Euler}),  see Proposition \ref{equival},  the proof is based in 
 Ito-Wentzell-Kunita formula  and stochastic analysis techniques. We point that in 
\cite{Flandoli} the authors show that the Lagrangian formulation   verifies necessarily the equation  (\ref{Euler}), for $d=3$, using the vorticity equation.  
We show that both formulations are 
equivalent for  any dimension.  Using this formulation we  prove a local in time existence  result for solutions in 
$C^{0} (\left[ 0,T\right]; ( H^{s}(\mathbb{T}^{d})^{d}))$ with $s> \frac{d}{2}+ 1$, new for equation (\ref{Euler}).

\section{Equivalent formulations.}

Let $\mathcal{L}_{\sigma^{j}}^{\ast}$ be the adjoint operator of $\mathcal{L}_{\sigma^{j}}$ with respect to the inner product in $L^{2}(\mathbb{T}^{d},\mathbb{R}^{d})$:
\begin{equation*}
	\langle\mathcal{L}_{\sigma^{j}}^{\ast}v,w\rangle_{L^{2}}=-\langle v,\mathcal{L}_{\sigma^{j}}w\rangle_{L^{2}}.
\end{equation*}
Since $\sigma^{j}$ is divergence free, we have
\begin{equation*}
	\mathcal{L}_{\sigma^{j}}^{\ast}v=\sigma^{j}\cdot\nabla v+\left(\nabla\sigma^{j}\right)^{\ast}v.
\end{equation*}

\begin{mydef}
	Given a (divergence free) velocity $u$, we define the material diferential  $\mathcal{D}$ by
	\begin{equation*}
		\mathcal{D}\gamma:=d\gamma+u\cdot\nabla\gamma dt+\sigma^{i}\cdot\nabla\gamma\circ dW_{t}^{i}
	\end{equation*}
\end{mydef}

Analogously to the deterministic case, see Proposition 2 in \cite{Pooley},   by simple calculation we have the identities

 \begin{equation}\label{cadeia}
 	\mathcal{D}\left(\gamma\beta\right)=\left(\mathcal{D}\gamma\right)\beta+\gamma\left(\mathcal{D}\beta\right), 
 \end{equation}

\begin{equation}\label{material}
	\mathcal{D}\nabla\gamma=\nabla\mathcal{D}\gamma -\left(\nabla u\right)^{\ast}\nabla\gamma dt-\left(\nabla\sigma^{i}\right)^{\ast}\nabla\gamma\circ dW_{t}^{i}.
\end{equation}

We use the notation $\mathbb{P}$ for the Leray-Hodge projection onto the space of divergence-free functions.\\

\begin{mypro}\label{equival} Assume that  $u$ is $C^{3, \alpha}$-continuos semimartingale.  Then $u$  is  solution of the equation 
(\ref{Euler}) if and only if verifies the Lagrangian formulation 

 \begin{eqnarray}
	\label{EDO}
		& dX_{t}=\sum_{j}\sigma^{j}(X_{t})\circ dW_{t}^{j}+u_{t}(X_{t})dt\\
		\label{5}
		& u_{t}(x)=\mathbb{P}\left[(\nabla A_{t})^{\ast}u_{0}(A_{t})\right](x),
\end{eqnarray}

where  $\ast$ means the transposition of matrices and denote the back-to-labels map $A$ by setting $A\left(\cdot, t\right)=X^{-1}\left(\cdot, t\right)$.
\end{mypro}

\begin{proof}
	$(\Rightarrow)$
	We have 
	\begin{equation}
		u(t,x)=u_{0}(x)-\int_{0}^{t}u\nabla uds-\sum_{j}\int_{0}^{t}\mathcal{L}^{\ast}_{\sigma^{j}}u\circ dW_{s}^{j}-\int_{0}^{t}\nabla pds
	\end{equation}
	and
	\begin{equation}
		X_{t}=x+\int_{0}^{t}u\left( s,X_{s}\right) ds+\sum_{j}\int_{0}^{t}\sigma^{j}\left( X_{s}\right)\circ dW_{s}^{j}.
	\end{equation}
 Then, from It\^o-Wentzell-Kunita formula, see Theorem 8.3 in  Chapter I of \cite{Ku2}, the $k$-th component of the process $u(t,X_{t})$ is given by

%	\begin{equation}
%		dX_{t}=u\left( t,X_{t}\right) dt+\sum_{j}\sigma^{j}\left( X_{t}\right)\circ dB^{j}_{t}
%	\end{equation}
%	or
%	\begin{equation}
%		dX^{i}_{t}=u^{i}\left( t,X_{t}\right) dt+\sum_{j}\sigma_{i}^{j}\left( X_{t}\right)\circ dB^{j}_{t}
%	\end{equation}
	\begin{align}
		\nonumber
		u^{k}\left( t,X_{t}\right)= & u_{0}\left( x\right)-\int_{0}^{t}\left( \left( u\cdot\nabla\right) u^{k}\left( s,X_{s}\right)+\left(\nabla p\right)\left( X_{s}\right)\right) ds\\
	\label{cancell}
		 &-\int_{0}^{t}\sum_{j}((\sigma^{j}\cdot\nabla) u^{k}\left( s,X_{s}\right)+\frac{\partial\sigma^{j}}{\partial x_{k}}\left( X_{s}\right) u\left( s,X_{s}\right))\circ dB^{j}_{s}\\
		 \nonumber
		 &+\sum_{i}\int_{0}^{t}\frac{\partial u^{k}}{\partial x_{i}}(s,X_{s})u^{i}(s,X_{s})ds+\sum_{i,j}\int_{0}^{t}\frac{\partial u^{k}}{\partial x_{i}}(s,X_{s}) \sigma_{i}^{j}(X_{s})\circ dB^{j}_{s}.
	\end{align}

	We  observe that 
	\begin{align*}
		\int_{0}^{t}(u\cdot\nabla)u^{k}(s,X_{s}) ds &=\int_{0}^{t}\sum_{i}u^{i}\frac{\partial u^{k}}{\partial x_{i}}\big|_{X_{s}}ds\\
		&=\sum_{i}\int_{0}^{t}u^{i}(s,X_{s})\frac{\partial u^{k}}{\partial x_{i}}(s, X_{s})ds\\
		&=\sum_{i}\int_{0}^{t}\frac{\partial u^{k}}{\partial x_{i}}(s,X_{s}) u^{i}(s,X_{s}) ds
	\end{align*}
	and
	\begin{align*}
		\int_{0}^{t}\sum_{j}(\sigma^{j}\cdot\nabla) u^{k}(s,X_{s})\circ dW_{s}^{j} &=\sum_{j}\int_{0}^{t}\sum_{i}\sigma_{i}^{j}\frac{\partial u^{k}}{\partial x_{i}}\big|_{X_{s}}\circ dW_{s}^{j}\\
		&=\sum_{i,j}\int_{0}^{t}\frac{\partial u^{k}}{\partial x_{i}}(s,X_{s})\sigma_{i}^{j}(X_{s})\circ dW_{s}^{j}.
	\end{align*}
	
	Making obvious cancellation we obtain
	\begin{equation*}
		u^{k}(t,X_{t})=u_{0}(x)-\int_{0}^{t}\left(\nabla p\right)(X_{s}) ds-\sum_{j}\int_{0}^{t}\frac{\partial\sigma^{j}}{\partial x_{k}}\left( X_{s}\right) u^{k}\left( X_{s}\right)\circ dW_{s}^{j},
	\end{equation*}
	i.e.,
	\begin{equation*}
		u(t,X_{t})=u_{0}(x)-\int_{0}^{t}\left(\nabla p\right)\left( X_{s}\right) ds-\sum_{j}\int_{0}^{t}\left(\nabla\sigma^{j}\right)^{\ast}u\big|_{X_{s}}\circ dW_{s}^{j}.
	\end{equation*}
	
		Now, we observe that   $\sum_{i}\frac{\partial X^{i}}{\partial x_k}u^{i}$ is the $k$-th coordinate of $\left(\nabla X\right)^{\ast}u$ and 
	\begin{equation*}
		\frac{\partial X^{i}}{\partial x_k}=\delta_{ik}+\int_{0}^{t}\nabla u^{i}(s,X_{s})\frac{\partial X_{s}}{\partial x_k}ds+\sum_{j}\int_{0}^{t}\left(\nabla\sigma_{i}^{j}\right)(X_{s})\frac{\partial X_{s}}{\partial x_k}\circ dW_{s}^{j}.
	\end{equation*}
Thus from It\^o's formula for the product of two semimartingales we deduce

	\begin{align*}
		\frac{\partial X^{i}}{\partial x_k}u^{i}\left( t,X_{t}\right) &=\delta_{ik}u_{0}(x)+\int_{0}^{t}u^{i}(s,X_{s})\nabla u^{i}(s,X_{s})\frac{\partial X_{s}}{\partial x_k} ds\\
		&\,\,\,+\int_{0}^{t}u^{i}(s,X_{s})\left(\sum_{j}\left(\nabla\sigma_{i}^{j}\right)\left( X_{s}\right)\frac{\partial X_{s}}{\partial x_k}\right)\circ dW_{s}^{j}-\int_{0}^{t}\frac{\partial X^{i}}{\partial x_k}\left(\nabla p\right)(X_{s}) ds\\
		&\,\,\,-\int_{0}^{t}\frac{\partial X^{i}}{\partial x_k}\left(\sum_{j}\left(\frac{\partial\sigma^{j}}{\partial x_i}\right)(X_{s}) u(s,X_{s})\right)\circ dW_{s}^{j}.
	\end{align*}
	Note that
	\begin{align*}
		\frac{\partial X^{i}}{\partial x_k}\left(\frac{\partial\sigma^{j}}{\partial x_i}\right) u &=\frac{\partial X^{i}}{\partial x_k}\left(\frac{\partial\sigma_{1}^{j}}{\partial x_i},\cdots ,\frac{\partial\sigma_{j}^{d}}{\partial  x_i}\right)\left( u_{1}, u_{2}, \cdots , u_{d}\right)^{\ast}\\
		&=\frac{\partial X^{i}}{\partial x_k}\sum_{\alpha}\frac{\partial\sigma_{\alpha}^{j}}{\partial x_i} u^{\alpha},
	\end{align*}
	hence $\sum_{i}\frac{\partial X^{i}}{\partial x_k}\left(\frac{\partial\sigma^{j}}{\partial x_i}\right) u=\sum_{i}\sum_{\alpha}\frac{\partial X^{i}}{\partial x_k}\frac{\partial\sigma_{\alpha}^{j}}{\partial x_i}u^{\alpha}$.\\

	On the other hand, we have 
	\begin{align*}
		u^{i}\left(\nabla\sigma_{i}^{j}\right)\frac{\partial X_{s}}{\partial x_k} &=u^{i}\left(\frac{\partial\sigma_{i}^{j}}{\partial x_1},\frac{\partial\sigma_{i}^{j}}{\partial x_2},\cdots ,\frac{\partial\sigma_{i}^{j}}{\partial x_d}\right)\left(\frac{\partial X^{1}}{\partial x_k},\cdots , \frac{\partial X^{d}}{\partial x_k}\right)^{\ast}\\
		&=u^{i}\sum_{\beta}\frac{\partial\sigma_{i}^{j}}{\partial\beta}\frac{\partial X^{\beta}}{\partial x_k},
	\end{align*}
	hence $\sum_{i}u^{i}\left(\nabla\sigma_{i}^{j}\right)\frac{\partial X_{s}}{\partial x_k}=\sum_{i}\sum_{\beta}u^{i}\frac{\partial\sigma_{i}^{j}}{\partial x_\beta}\frac{\partial X^{\beta}}{\partial x_k}$.\\

	Thus we obtain

	\begin{equation*}
		\sum_{i}\int_{0}^{t}u^{i}(s,X_{s})\left(\sum_{j}\left(\nabla\sigma_{i}^{j}\right)\left( X_{s}\right)\frac{\partial X_{s}}{\partial x_k}\right)\circ dW_{s}^{j}=\sum_{i}\int_{0}^{t}\frac{\partial X^{i}}{\partial x_k}\left(\sum_{j}\left(\frac{\partial\sigma^{j}}{\partial x_i}\right)(X_{s}) u(s,X_{s})\right)\circ dW_{s}^{j}
	\end{equation*}
	Then the $k-$th term of $\left(\nabla X_{t}\right)^{\ast}u\left( t,X_{t}\right)$ is
	\begin{equation*}
		\sum_{i}\frac{\partial X^{i}}{\partial x_k}u^{i}\left( t,X_{t}\right)=\sum_{i}\delta_{ik}u_{0}(x)-\sum_{i}\int_{0}^{t}\frac{\partial X^{i}}{\partial x_k}\left(\nabla p\right)(X_{s}) ds+\sum_{i}\int_{0}^{t}u^{i}\left( s,X_{s}\right)\nabla u^{i}\left( s,X_{s}\right)\frac{\partial X_{s}}{\partial x_k} ds,
	\end{equation*}
	i.e.,
	\begin{align*}
		\left(\nabla X_{t}\right)^{\ast}u\left( t,X_{t}\right)&=u_{0}(x)+\int_{0}^{t}\left(\nabla X_{s}\right)^{\ast}\left(\nabla u(s,X_{s})\right)^{\ast}u\left( s,X_{s}\right) ds-\int_{0}^{t}\left(\nabla X_{s}\right)^{\ast}\left(\nabla p\right)(X_{s}) ds\\
		&=u_{0}(x)+\frac{1}{2}\int_{0}^{t}\nabla\left(\vert u_{s}\vert^{2}\circ X_{s}\right) ds-\int_{0}^{t}\nabla\left( p\circ X_{s}\right) ds\\
		&=u_{0}+\nabla\tilde{q},
	\end{align*}
	where $\tilde{q}:=\int_{0}^{t}\left[\frac{1}{2}\left(\vert u_{s}\vert^{2}\circ X_{s}\right)-p\circ X\right] ds$.\\

	Then, if we denote $M_{t}:=\left(\nabla X_{t}\right)^{\ast}$, follow that 
	\begin{align*}
		u_{t}\circ X_{t}=u\left( t,X_{t}\right)&=M_{t}^{-1}u_{0}+M_{t}^{-1}\nabla\tilde{q}\\
		&=(\nabla A_{t}\big|_{X_{t}})^{\ast}u_{0}+(\nabla A_{t}\big|_{X_{t}})^{\ast}\nabla\tilde{q}.
	\end{align*}

	Finally we conclude 
	
	\begin{align*}
		u_{t}&=(\nabla A_{t})^{\ast}u_{0}\circ A_{t}+(\nabla A_{t})^{\ast}(\nabla\tilde{q})(A_{t})\\
		&=(\nabla A_{t})^{\ast}u_{0}\circ A_{t}+\nabla q,
	\end{align*}
	where $q:=\tilde{q}\circ A$. Therefore, $u_{t}=\mathbb{P}\left[\left(\nabla A_{t}\right)^{\ast}u_{0}\circ A_{t}\right]$.\\

	$(\Leftarrow)$  We follow directly from proposition 2.2 in  \cite{Flandoli2}. We shall also show  an alternative formal proof.	
	We set $v=u_{0}\circ A$, then by Theorem 2.3.2 of  \cite{Chow} we have $\mathcal{D}A=0$ and $\mathcal{D} v=0$.\\
Since $u$ satisfies \eqref{5}  there exists a function $q$ such that
\[u=\left(\nabla A\right)^{\ast}v-\nabla q.
\]
Then by \eqref{cadeia} and \eqref{material} we have 
	\begin{align*}
		\mathcal{D}u &=\mathcal{D}\left[\left(\nabla A\right)^{\ast}v-\nabla q\right]\\
		&=\mathcal{D}\left[\left(\nabla A\right)^{\ast}v\right]-\mathcal{D}\nabla q\\
		&=\left[\mathcal{D}\left(\nabla A\right)^{\ast}\right] v+\left(\nabla A\right)^{\ast}\mathcal{D}v-\mathcal{D}\nabla q\\
		&=\left[\nabla\mathcal{D}A-\left(\nabla u\right)^{\ast}\left(\nabla A\right) dt-\left(\nabla\sigma^{i}\right)^{\ast}\nabla A\circ dW_{t}^{i}\right]^{\ast} v\\
		&\,\,\,\,\,\,\,-\nabla\mathcal{D} q+\left(\nabla u\right)^{\ast}\nabla qdt+\left(\nabla\sigma^{i}\right)^{\ast}\nabla q\circ dW_{t}^{i}.
\end{align*}
Hence, after a calculation and a rearrangement of the terms, we get 
\begin{align*}
		\mathcal{D}u &=-\left(\nabla u\right)^{\ast}\left(\nabla A\right)^{\ast}vdt+\left(\nabla u\right)^{\ast}\nabla qdt-\nabla\mathcal{D}q\\
		&\,\,\,\,\,\,\,\,-\left(\nabla\sigma^{i}\right)^{\ast}\left[\left(\nabla A\right)^{\ast}v-\nabla q\right]\circ dW_{t}^{i}\\
		&=-\left(\nabla u\right)^{\ast}udt-\nabla\mathcal{D}q-\left(\nabla\sigma^{i}\right)^{\ast}u\circ dW_{t}^{i}\\
		&=-\nabla\frac{|u|^{2}}{2}dt-\nabla\mathcal{D}q-\left(\nabla\sigma^{i}\right)^{\ast}u\circ dW_{t}^{i}.
	\end{align*}
	Then,
	\begin{align*}
		u(t,x)&=u_{0}(x)-\int_{0}^{t}u\cdot\nabla uds-\int_{0}^{t}\sigma^{i}\cdot\nabla u\circ dW_{s}^{i}-\nabla\int_{0}^{t}\frac{\vert u\vert^{2}}{2}ds-\nabla 
		\int_{0}^{t} \mathcal{D}q_{s}
	-\int_{0}^{t}\left(\nabla\sigma^{i}\right)^{\ast}u\circ dW_{s}^{i}\\
		&=u_{0}(x)-\int_{0}^{t}u\cdot\nabla uds-\int_{0}^{t}\mathcal{L}_{\sigma^{i}}^{\ast}u\circ dW_{s}^{i}- \int_{0}^{t} \nabla p dt ,
	\end{align*}
	
	where formally we have
	
	\begin{equation*}
		p=\frac{|u|^{2}}{2} + \partial_{t}q + u\cdot\nabla q+\sigma^{i}\cdot\nabla q \ \partial_{t}W_{t}^{i}
 .
	\end{equation*}

	Therefore we conclude that $u$ is solution of the Euler equation (\ref{Euler}).

\end{proof}

\section{An Existence Theorem}

\subsection{  Decomposition of the  Flow }

We use the idea of \cite{Flandoli} to decompose  the stochastic flow in the system \eqref{EDO}-\eqref{5}. More precisely, we consider   the  stochastic equation without drift:
\begin{equation*}
	d\varphi_{t}=\sum_{k}\sigma_{k}\left(\varphi_{t}\right)\circ dW_{t}^{k},\,\,\,\, \varphi_{0}=I,
\end{equation*}	
where $I$ is the identity diffeomorphism of $\mathbb{T}^{d}$. Under the assumption that $\sum_{k}\|\sigma_{k}\|_{5,\alpha^{\prime}}^{2}<\infty$ for $\alpha^{\prime}\in(0,1)$, the above equation generates a stochastic flow $\left\{\varphi_{t}\right\}_{t\geq 0}$ of $C^{4,\alpha}-$diffeomorphisms on $\mathbb{T}^{d}$, where $\alpha\in(0,\alpha^{\prime})$.\\
We denote by $\omega$ a generic random element in a probability space $\Omega$. For a given random vector field $u:\Omega\times\left[ 0,T\right]\times\mathbb{T}^{d}\to\mathbb{R}^{d}$, we define
\begin{equation}
	\tilde{u}\left(\omega, x\right)=\left[\left(\varphi_{t}\left(\omega,\cdot\right)^{-1}\right)_{\ast} u_{t}\left(\omega,\cdot\right)\right]\left( x\right)
\end{equation}
which is the pull-back of the field $u_{t}\left(\omega,\cdot\right)$ by the stochastic flow $\left\{\varphi_{t}\left(\omega,\cdot\right)\right\}_{t\geq 0}$. If we denote by $K_{t}\left(\omega, x\right)=\left(\nabla\varphi_{t}\left(\omega, x\right)\right)^{-1}$, i.e., the inverse of the Jacobi matrix, then
\begin{equation}
	\tilde{u}\left(\omega, x\right)=K_{t}\left(\omega, x\right) u_{t}\left(\omega,\varphi_{t}\left(\omega, x\right)\right).
\end{equation}

From this expression we see that if $u\in C^{0}\left(\left[ 0,T\right], H^{s}\right)$ a.s., then one also a.s. $\tilde{u}\in C^{0}\left(\left[ 0,T\right], H^{s}\right)$. Moreover, if the process $u$ is adapted, then so is $\tilde{u}$. Now we consider the random ODE
\begin{equation}
	\dot{Y}_{t}=\tilde{u}_{t}\left( Y_{t}\right),\,\,\,\,\, Y_{0}=I.
\end{equation}
Applying the It\^o-Wentzell-Kunita formula, we see that (cf. \cite{Flandoli})
\begin{equation*}
	X_{t}=\varphi_{t}\circ Y_{t}
\end{equation*}
is the flow of $C^{3,\alpha}-$diffeomorphisms associated to the SDE in \eqref{EDO}-\eqref{5}.\\

%The Eulerian-Lagrangian form of the Euler equations comprises the following system:
%\begin{equation}\label{4}
%	\partial_{t} B+\left( \tilde{u}\cdot\nabla\right) B=0,
%\end{equation}
%\begin{equation}\label{5}
%	u=\mathbb{P}\left[\left(\nabla A\right)^{\ast} v\right],
%\end{equation}
%\begin{equation}\label{6}
%	\partial_{t} v+\left( u\cdot\nabla\right) v=0.
%\end{equation}
%Given an initial divergence-free velocity $u_{0}$ for the classic equations, we choose initial conditions for the above system as follows:
%\begin{equation}\label{7}
%	A\left( x, 0\right)=x,
%\end{equation}
%\begin{equation}\label{8}
%	u\left( x,0\right)=v\left( x,0\right)=u_{0}\left( x\right).
%\end{equation}

\subsection{ Sobolev Estimations.}

All notations and results in this subsection we follow from \cite{Pooley}. For $r\geq 0$, we will use the notation $H^r$ variously for sacalar or vector valued functions in $H^{r}\left(\mathbb{T}^{d}\right)$ (componentwise), where this does not cause ambiguity. We will often consider functions in spaces of the norm $C^{0} (\left[ 0,T\right]; ( H^{s}(\mathbb{T}^{d})^{d}))$.\\
To simplify notation we define $\Sigma_{s}\left( T\right)$ (usually denoted $\Sigma_{s}$) for $T\geq 0$ and $s\geq 0$ by
\[
\Sigma_{s}\left( T\right):=C^{0}(\left[ 0,T\right];( H^{s}(\mathbb{T}^{d}))^{d}).
\]
We consider the natural norm on $\Sigma_{s}$:
\[
\| u\|_{\Sigma_{s}}=\sup_{t\in[0,T]}\| u(t)\|_{s}
\]

We begin by stating two inequalities concerning the advection term $(u\cdot\nabla)v$, using the notation $B(u, v):=(u\cdot\nabla)v$. The following two results are taken from Lemma 1 and Lemma 2 of \cite{Pooley}.
\begin{mylem}
For $s>\frac{d}{2}$ there exists $C_{1}>0$ such that if $u\in H^{s}$ and $v\in H^{s+1}$ then $B\left( u, v\right)\in H^{s}$ and
	\[
	\| B\left( u,v\right)\|_{s}\leq C_{1}\| u\|_{s}\| v\|_{s+1}.
	\]
	\end{mylem}
\begin{mylem}\label{lemma2}
	If $s>\frac{d}{2}+1$ there exists $C_{2}>0$ such that for $u\in H^{s}$, $v\in H^{s+1}$ with divergence-free we have
	\[
	\lvert \left( B\left( u,v\right),v\right)_{s}\vert\leq C_{2}\| u\|_{s}\| v\|_{s}^{2}
	\]
\end{mylem}
We use the following shorthand for closed balls in $\Sigma_{s}$:
\[
B_{M}=\overline{B_{\|\cdot\|_{\Sigma_{s}}}\left( 0,M\right)},
\]
i.e., $B_{M}$ is the closed unit ball centred at the origin of radius $M>0$ with respect to the norm $\|\cdot\|_{\Sigma_{s}}$. Where ambiguity could arise we write $B_{M}\left( T\right)$ for the closed ball in $\Sigma_{s}\left( T\right)$.\\

We need the following key technical result, see Lemma 3 of \cite{Pooley}.

\begin{mylem}\label{lemma3}
	If $s>\frac{d}{2}+1$ and $\eta,\, v\in\Sigma_{s}\left( T\right)$ then $\mathbb{P}\left[\left(\nabla\eta\right)^{\ast} v\right]\in\Sigma_{s}$ and there exists a constant $C_{3}>0$ (independent of $\eta,\, v\, t$ and $T$) such that for fixed $t$,
 \begin{equation}\label{eq1lemma3}
 	\|\mathbb{P}\left[\left(\nabla\eta\right)^{\ast} v\right]\|_{r}\leq C_{3}\| u\|_{s}\| v\|_{r},
 \end{equation}
where $r=s$ or $r=s-1$. Furthermore, there exists $C_{3}^{\prime}>0$ such that for any $M>0$ and $T>0$, the following bounds hold uniformly with respect to $t\in\left[ 0, T\right]$ for any $\eta_{1},\, \eta_{2},\, v_{1},\, v_{2}\in B_{M}\left( T\right)$:
\begin{equation}\label{eq2lemma3}
	\|\mathbb{P}\left[\left(\nabla\eta_{1}\right)^{\ast} v_{1}-\left(\nabla\eta_{2}\right)^{\ast} v_{2}\right]\|_{X}\leq C_{3}^{\prime}\left(\|\eta_{1}-\eta_{2}\|_{X}+\| v_{1}-v_{2}\|_{X}\right),
\end{equation}
where $X$ is $L^{2}\left(\mathbb{T}^{d}\right)$ or $H^{s-1}$.
\end{mylem}
The next lemma gives uniform bounds on the $H^{s}$ norms of solutions to the transport equations \eqref{4} and \eqref{6}. We will consider the following system:
\begin{equation}
	\label{23}
	\begin{cases}
		& \partial_{t}f+\left( u\cdot\nabla\right) f=g,\\ 
		& f(0)=f_{0}.
	\end{cases}
\end{equation}
where $f,\, g:[0,T]\times\mathbb{T}^{d}\to\mathbb{T}^{d}$ and $u$ is divergence-free.
\begin{mylem}\label{lemma4}
	Let $s>\frac{d}{2}+1$  and fix $f_{0}\in H^{s},\, g\in\Sigma_{s}$. If $u\in\Sigma_{s}$ is non-zero and divergence free then exists a unique solution $f$ to \eqref{23}. Furthermore, the solution $f\in\Sigma_{s}\cap C^{1}([0,T]; H^{s-1})\cap C^{1}([0, T]\times\mathbb{T}^{d})$ and there exists $C_{4}>0$ (from Lemma \ref{lemma2}) such that if $r,\, t\in [0,T]$ we have:
	\begin{equation}\label{eqlemma4}
		\| f(t)\|_{s}\leq\left(\| f(r)\|_{s}+\frac{\| g\|_{\Sigma_{s}}}{C_{4}\| u\|_{\Sigma_{s}}}\right)\exp\left( C_{4}\vert t-r\vert\| u\|_{\Sigma_{s}}\right)-\frac{\| g\|_{\Sigma_{s}}}{C_{4}\| u\|_{\Sigma_{s}}}.
	\end{equation}
\end{mylem}

\begin{mylem}\label{lemma5}
	For $s>\frac{d}{2}+1$ fix $u_{1},\, u_{1}\in\Sigma_{s}$ and $f_{0}\in H^{s}$. Let $g_{1}=g_{2}=0$ or $g_{i}=-u_{i}$ for $i=1,\, 2$. If $f_{1},\, f_{2}$ are the solutions of \eqref{23} corresponding to $u_{1},\, u_{2},\, g_{1},\, g_{2}$ respectively, then in the case that $g_{1}=g_{2}=0$, there exists $C_{5}>0$ depending only of $s$ such that
	\begin{equation}\label{29}
		\| f_{1}(t)-f_{2}(t)\|_{L^{2}}\leq C_{5}\| f_{1}+f_{2}\|_{\Sigma_{s}}\| u_{1}-u_{2}\|_{\Sigma_{0}}t
	\end{equation}
for all $t\in[0, T]$. In the case that $g_{i}=-u_{i}$ for $i=1,\, 2$ we instead have
\begin{equation}
		\| f_{1}(t)-f_{2}(t)\|_{L^{2}}\leq \left( C_{5}\| f_{1}+f_{2}\|_{\Sigma_{s}}+1\right)\| u_{1}-u_{2}\|_{\Sigma_{0}}t
\end{equation}
\end{mylem}

We consider  the following  transport equation:
\begin{equation}\label{4}
	\partial_{t} B+\left( \tilde{u}\cdot\nabla\right) B=0,
\end{equation}
%\begin{equation}\label{5}
%	u=\mathbb{P}\left[\left(\nabla A\right)^{\ast} v\right],
%\end{equation}
\begin{equation}\label{6}
	\partial_{t} v+\left( \tilde{u}\cdot\nabla\right) v=0.
\end{equation}
Given an initial divergence-free velocity $u_{0}$ for the classic equations, we choose initial conditions for the above system as follows:
\begin{equation}\label{7}
	B\left( x, 0\right)=x,
\end{equation}
\begin{equation}\label{8}
	u\left( x,0\right)=v\left( x,0\right)=u_{0}\left( x\right).
\end{equation}

 Also  $\eta\left( x,t\right):=B\left( x,t\right)-x$ and replace \eqref{4} and \eqref{7} with the equations
\begin{equation}\label{90}
	\partial_{t}\eta+\left( \tilde{u}\cdot\nabla\right)\eta+\tilde{u}=0,
\end{equation}
\begin{equation}\label{91}
	\eta\left( x, 0\right)=0
\end{equation}
respectively. We do this because the identity map (hence $B$) does not have sufficient Sobolev regularity when considered as a function on the torus with values in $\mathbb{R}^{d}$ (i.e. whitout accounting for the topology of the target torus).\\
%The following proposition encapsulates the derivation of \eqref{5} (sometimes called the Weber formula) which can be found in \cite{Constantin}.
%\begin{mypro}
%	Let $d\geq 2$, consider $u\in C^{1}\left(\left( 0, T\right)\times\mathbb{T}^{d}\right)$, with $u\left( 0\right)\in C^{1}\left(\mathbb{T}^{d}\right)$. If $u$ is divergence-free and satisfies
%	\[
%	\frac{\partial u}{\partial t}+\left( u\cdot\nabla\right) u+\nabla p=0
%	\]
%	for some $p$, with spattially periodic boundary conditions then $A\in C^{1}\left(\left( 0,T\right)\times\mathbb{T}^{d};\mathbb{T}^{d}\right)$ and $u$ satisfies \eqref{5} with $v\left( x,t\right)=u_{0}\left( A\left( x,t\right)\right)$
%\end{mypro}

\subsection{Contraction }

The aim of the rest of this paper is to prove the following theorem.
\begin{myth}
	If $d\geq 2$, $s>\frac{d}{2}+1$ and $u_0\in H^s$ is divergence free then there exists $T(\omega)>0$, such that the systems \eqref{4}-\eqref{5}  has  solution   $u\in\Sigma_{s}(T)$
\end{myth}

\begin{proof}
%\in C^{0}\left([0,T], C^{l+1,\alpha}\left(\mathbb{T}^{d}, \mathbb{R}^{d}\right)\right)
 For $u\in\Sigma_{s}$, we consider the following system:

\begin{equation}
\label{eqn:1.1}
\begin{cases}
 & \tilde{u}_{t}\left( x\right)=\left[\left(\varphi_{t}^{-1}\right)_{\ast}u_{t}\right]\left( x\right),\\ 
 & \dot{Y}_{t}=\tilde{u}_{t}\left( Y_{t}\right),\,\,\,\,\, Y_{0}=I,\\
 & X_{t}=\varphi_{t}\left( Y_{t}\right),\\
 & S{u}_{t}\left( x\right)=\mathbb{P}\left[\left(\nabla X_{t}^{-1}\right)^{\ast} u_{0}\left( X_{t}^{-1}\right)\right]\left( x\right).
\end{cases}
\end{equation}

From  Chapter 7 of \cite{Behzadan}  we have that exists positives constants $C_{0}:=C_{0}(t, \varphi)$, $C_{1}:=C_{1}(t, \varphi)$, $C_{2}:=C_{2}(t, \varphi)$ such that 
\begin{eqnarray}
	\label{c1}
	&\|\varphi^{-1}\|_{s}\leq C_{1};\\
	\label{c0}
	& \|\tilde{u}\|_{s}=\| \left(\nabla\varphi^{-1}\right)(\varphi)u(\varphi)\|_{s}\leq C_{0}\| u\|_{s};\\
	\label{c2}
	& \|\eta\circ\varphi^{-1}\|_{s}\leq C_{2}\|\eta\|_{s}\hspace{0,2cm}\text{and}\hspace{0,2cm}\| v\circ\varphi^{-1}\|_{s}\leq C_{2}\| v\|_{s}.
\end{eqnarray}

Fix $s>\frac{d}{2}+1$ and let $C_3$, $C_4$ be the constants in \eqref{eq1lemma3}, \eqref{eqlemma4} (from Lemmas \ref{lemma3} and \ref{lemma4}) respectively. Fix $\| u_{0}\|_{s}<M$ and $T>0$ so that
\begin{equation*}
	C\| u_{0}\|_{s}\exp\left( C_{4}C_{0}TM\right)\left[\frac{C_{2}^{2}}{C_{1}C_{4}}\left(\exp\left( C_{4}C_{0}TM\right)-1\right)+1\right]\leq M,
\end{equation*}
where C is a constant to be defined later.\\
Let $u\in B_{M}(T)$ be a divergence free function and let $\eta=Y_{t}^{-1}-x$ is unique solution of \eqref{90} with initial data $\eta_{0}=0$. Let $v=u_{0}(Y_{t}^{-1})$
 is the unique solution of \eqref{6} for initial data $v_0=u_0$.
%Let $u\in B_{M}(T)$ be a divergence free function and let $\eta$ be the solution of \eqref{23} for the flow $\tilde{u}$ with initial data $\eta_{0}=0$ and forcing $g=-\tilde{u}$. Let $v$ be the solution of \eqref{23} for initial data $v_{0}=u_{0}$ with $g=0$.\\

From \eqref{c0} follow that
\begin{equation*}
	\| \tilde{u}\|_{\Sigma_{s}}=\sup_{t\in[0,T]}\| \tilde{u}\|_{s}\leq C_{0}M,
\end{equation*}
Then by Lemma \ref{lemma4},
% Let $C_{0}$, $C_{1}$, $C_{2}$ constants arising from the application of some results in Sobolev spaces, see for thats results in \cite{Behzadan}, then by Lemmas \ref{lemma3} and \ref{lemma4},

\begin{equation}\label{v}
	\| v(t)\|_{s}\leq\| u_{0}\|_{s}\exp\left( C_{4}t\|\tilde{u}\|_{\Sigma_s}\right)\leq\| u_{0}\|_{s}\exp\left( C_{0}C_{4}TM\right)
\end{equation}
and
\begin{equation}\label{eta}
	\|\eta(t)\|_{s}\leq\frac{1}{C_{4}}\left(\exp\left( C_{0}C_{4}TM\right)-1\right).
\end{equation}
Hence, with $C=C_{1}C_{2}C_{3}$, (the constants $C_1$, $C_2$, $C_3$ are given by \eqref{c1}, \eqref{c2} and Lemma \ref{lemma3} respectively)
\begin{align*}
	\| Su(t)\|_{s}&\leq\|\mathbb{P}\left[(\nabla(\eta\circ\varphi^{-1}))^{\ast}v(\varphi^{-1})\right]\|_{s}+\|\mathbb{P}\left[(\nabla\varphi^{-1})^{\ast}v(\varphi^{-1})\right]\|_{s}\\
	&\leq C_{3}\|\eta\circ\varphi^{-1}\|_{s}\| v(\varphi^{-1})\|_{s}+C_{3}\|\varphi^{-1}\|_{s}\| v(\varphi^{-1})\|_{s}\\
	&\leq C_{3}C_{2}^{2}\|\eta\|_{s}\| v\|_{s}+C_{3}C_{1}C_{2}\| v\|_{s}\\
	&=C_{3}C_{2}\| v\|_{s}\left( C_{2}\|\eta\|_{s}+C_{1}\right),
\end{align*}
here the third inequality follows from \eqref{c1} and \eqref{c2}. Then by \eqref{v} and \eqref{eta}
\begin{align}\nonumber
	\| Su(t)\|_{s}&\leq \| u_{0}\|_{s}\exp\left( C_{4}C_{0}tM\right)\left[\frac{C_{3}C_{2}^{2}}{C_{4}}\left(\exp\left( C_{4}C_{0}tM\right)-1\right)+C_{3}C_{1}C_{2}\right]\\
	\label{bola}
	&=C\| u_{0}\|_{s}\exp\left( C_{4}C_{0}tM\right)\left[\frac{C_{2}^{2}}{C_{1}C_{4}}\left(\exp\left( C_{4}C_{0}tM\right)-1\right)+1\right]\leq M
\end{align}
for all $t\in[0,T]$. Hence $S:B_{M}(T)\to B_{M}(T)$.\\

Before  proving that the map $S$ is a contraction in a certain space, we need to prove a couple of inequalities.  By Lemmas \ref{lemma3}, \ref{lemma4} and \ref{lemma5} we have 

\begin{align*}
	&\left\|\mathbb{P}\left[\left(\nabla\left(\eta_{1}\circ\varphi^{-1}\right)\right)^{\ast}v_{1}\left(\varphi^{-1}\right)-\left(\nabla\left(\eta_{2}\circ\varphi^{-1}\right)\right)^{\ast}v_{2}\left(\varphi^{-1}\right)\right]\right\|_{L^{2}}\\
	&\hspace{1,5cm}\leq C_{3}M\left[\left\|\eta_{1}\circ\varphi^{-1}-\eta_{2}\circ\varphi^{-1}\right\|_{L^2}+\left\| v_{1}\circ\varphi^{-1}-v_{2}\circ\varphi^{-1}\right\|_{L^2}\right]\\
	&\hspace{1,5cm}= C_{3}M\left[\left\|\eta_{1}-\eta_{2}\right\|_{L^2}+\left\| v_{1}-v_{2}\right\|_{L^2}\right]\\
	&\hspace{1,5cm}\leq C_{3}M\left[\left( C_{5}\left\|\eta_{1}+\eta_{2}\right\|_{\Sigma_s}+1\right)\left\|\tilde{u}_{1}-\tilde{u}_{2}\right\|_{\Sigma_0}t+C_{5}\left\| v_{1}-v_{2}\right\|_{\Sigma_{s}}\left\|\tilde{u}_{1}-\tilde{u}_{2}\right\|_{\Sigma_0}t\right]\\
	&\hspace{1,5cm}\leq C_{3}M\left[\left(\frac{2C_5}{C_4}\left(\exp\left( C_{4}C_{0}tM\right)-1\right)+1\right)t\left\|\tilde{u}_{1}-\tilde{u}_{2}\right\|_{\Sigma_0}+2C_{5}Mt\left\| u_{0}\right\|_{s}\exp\left( C_{4}C_{0}tM\right)\left\|\tilde{u}_{1}-\tilde{u}_{2}\right\|_{\Sigma_0}\right]\\
	&\hspace{1,5cm}\leq C_{3}M\left[\left(\frac{2C_{5}}{C_{4}}\left(\exp\left( C_{4}C_{0}TM\right)-1\right)+1\right) t+2C_{5}t\left\| u_{0}\right\|_{s}\exp\left(C_{4}C_{0}TM\right)\right]\left\|\tilde{u}_{1}-\tilde{u}_{2}\right\|_{\Sigma_{0}},
\end{align*}
and
\begin{align*}
	\left\|\mathbb{P}\left[\left(\nabla\varphi^{-1}\right)^{\ast}\left( v_{1}(\varphi^{-1})-v_{2}(\varphi^{-1})\right)\right]\right\|_{L^{2}}&\leq C_{3}M\left\| v_{1}(\varphi^{-1})-v_{2}(\varphi^{-1})\right\|_{L^2}\\
	&=C_{3}M\left\| v_{1}-v_{2}\right\|_{L^2}\\
	&\leq C_{3}C_{5}M\left\| v_{1}-v_{2}\right\|_{\Sigma_{s}}\left\|\tilde{u}_{1}-\tilde{u}_{2}\right\|_{\Sigma_{0}}t\\
	&\leq 2C_{3}C_{5}Mt\left\| u_{0}\right\|_{s}\exp\left( C_{4}tC_{0}M\right)\left\|\tilde{u}_{1}-\tilde{u}_{2}\right\|_{\Sigma_{0}},
\end{align*}
where $C_{3}$, $C_{4}$, $C_{5}$ are the constants from Lemmas \ref{lemma3}, \ref{lemma4} and \ref{lemma5} respectively.\\

Now we will show that the map $S$ is a contraction on $B_{M}(T)$ in the $L^{2}$ norm if $T$ is sufficiently small. For $u_{1},\, u_{2}\in B_{M}(T)$ we construct $v_{i}$ and $\eta_{i}$ from $u_{i}$ as above for $i=1,\,2$ with $v_{1}(\cdot,0)=v_{2}(\cdot,0)=u_{0}$. Then by the inequalities above
\begin{align}\nonumber
	\left\| Su_{1}-Su_{2}\right\|_{L^{2}} &\leq\left\|\mathbb{P}\left[\left(\nabla\left(\eta_{1}\circ\varphi^{-1}\right)\right)^{\ast}v_{1}\left(\varphi^{-1}\right)-\left(\nabla\left(\eta_{2}\circ\varphi^{-1}\right)\right)^{\ast}v_{2}\left(\varphi^{-1}\right)\right]\right\|_{L^{2}}\\
	\nonumber
	&\hspace{4cm}+\left\|\mathbb{P}\left[\left(\nabla\varphi^{-1}\right)^{\ast}\left( v_{1}(\varphi^{-1})-v_{2}(\varphi^{-1})\right)\right]\right\|_{L^{2}}\\
	\label{contraction}
	&\leq C_{3}M\left(\frac{2C_{5}}{C_{4}}\left(\exp\left( C_{4}C_{0}TM\right)-1\right)+1\right) t\left\|\tilde{u}_{1}-\tilde{u}_{2}\right\|_{\Sigma_{0}}\\
	\nonumber
	&\hspace{4cm}+4C_{3}C_{5}Mt\left\| u_{0}\right\|_{s}\exp\left(C_{4}C_{0}TM\right)\left\|\tilde{u}_{1}-\tilde{u}_{2}\right\|_{\Sigma_{0}}\\
	\nonumber	
	&\leq 2TMC_{3}C_{6}\left(C_{5}\left(2\|u_{0}\|_{s}+\frac{1}{C_{4}}\right)\exp\left(C_{4}C_{0}TM\right)+\frac{1}{2}-\frac{C_{5}}{C_{4}}\right)\left\| u_{1}-u_{2}\right\|_{\Sigma_{0}}\\
	\nonumber
	&= C(u_{0},\,\omega, \, M,\, T)\left\| u_{1}-u_{2}\right\|_{\Sigma_{0}},
\end{align}
where lasts  inequality follows from a change of variables and H\"older's inequality
\begin{align*}
	\left\|\tilde{u}_{1}-\tilde{u}_{2}\right\|_{\Sigma_0}&=\sup_{t\in[0,T]}\left\|\tilde{u}_{1}-\tilde{u}_{2}\right\|_{L^2}=\sup_{t\in[0,T]}\left\|\left(\nabla\varphi\right)^{-1}u_{1}(\varphi)-\left(\nabla\varphi\right)^{-1}u_{2}(\varphi)\right\|_{L^2}\\
	&=\sup_{t\in[0,T]}\left\|\left(\nabla\varphi^{-1}\right)(\varphi)\left( u_{1}-u_{2}\right)(\varphi)\right\|_{L^2}\\
	&=\sup_{t\in[0,T]}\left\|\left(\nabla\varphi^{-1}\right)\left( u_{1}-u_{2}\right)\right\|_{L^2}\leq\sup_{t\in[0,T]}\left(\left\|\nabla\varphi^{-1}\right\|_{L^\infty}\left\| u_{1}-u_{2}\right\|_{L^2}\right)\\
	&\leq C_{6}\left\|u_{1}-u_{2}\right\|_{\Sigma_{0}},
\end{align*}
where $C_{6}:=\sup_{t\in[0,T]}\left\|\nabla\varphi^{-1}\right\|_{L^\infty}$.\\

We observe that  $C( u_{0},\,\omega,\, M,\, T)$ is given by the formula
\begin{equation}\label{constante}
	C( u_{0},\,\omega,\, M,\, T):=2T\left[ C_{5}MC_{3}C_{6}\left( 2\left\| u_{0}\right\|_{s}+\frac{1}{C_{4}}\right)\exp\left(C_{4}C_{0}TM\right)+C_{3}C_{6}M\left(\frac{1}{2}-\frac{C_{5}}{C_{4}}\right)\right],.
\end{equation}

 We observe that   $\varphi\in C^{0}([0,T],C^{4,\alpha}(\mathbb{T}^{d},\mathbb{T}^{d}))$ then this  implies that $C_0, C_1, C_2$ and $C_6$ are limited  uniformly in time.
Then taking the supreme of \eqref{contraction} with respect to $t$ and choosing $T>0$ small enough, we see that $S$ is a contraction. 
We finished the proof as in  \cite{Pooley}.
We conclude that $S$ has a unique accumulation point $u$, in the closure of $B_M$ with respect to $\| \cdot \|_{\Sigma_{0}}$. Since $B_M(T)$ is convex and closed in $\Sigma_{s}$ it is weakly closed, hence $u\in B_M$ is a fixed point of $S$. 

\end{proof}

%%%%%%%%%%%%%%%%%%%%%%%%
\section*{Acknowledgements}
%%%%%%%%%%%%%%%%%%%%%%%%
This study was financed in part by the Coordenação de
Aperfeiçoamento de Pessoal de Nível Superior - Brasil (CAPES) - Finance Code 001- on behalf of the first-named author. The second author is funded by the grant $2020/15691-2$ and $2020/04426-6$, São Paulo Research Foundation (FAPESP)

%OBS: segui os indicações do site da CAPES https://www.ppgcc.ufscar.br/pt-br/news/obrigatoriedade-da-indicacao-de-agradecimento-a-capes-e-demais-agencias-de-fomento-em-artigos-publicados#:~:text=Para%20fins%20de%20identifica%C3%A7%C3%A3o%20da,para%20todos%20os%20financiamentos%20recebidos.

%%%%%%%%%%%%%%%%%

\end{document}